\documentclass[11pt]{article}

\usepackage[margin=1in]{geometry}
\usepackage{amsmath,amssymb,amsthm,mathtools}
\usepackage{needspace}
\usepackage[hidelinks]{hyperref}

\newtheorem{theorem}{Theorem}[section]
\newtheorem{lemma}[theorem]{Lemma}
\newtheorem{corollary}[theorem]{Corollary}

\newcommand{\one}{\mathbf 1}
\newcommand{\ii}{\mathrm i}

\title{A six-neuron counterexample to the target-free clique conjecture}
\author{Jesse Geneson}
\date{}

\begin{document}

\maketitle

\begin{abstract}
The target-free clique conjecture asserts that the supports of stable fixed
points of a nondegenerate combinatorial
threshold-linear network (CTLN) are exactly its target-free cliques:
bidirected cliques for which no outside vertex receives an edge from every
clique vertex.  We give an explicit six-neuron counterexample.  For every
sufficiently small $\varepsilon>0$, the CTLN
defined by one fixed graph at $\delta=29\varepsilon/25$ is
nondegenerate and has a stable fixed point with
nonclique full support.  Its values of
$q=\delta(1-\varepsilon)/\varepsilon$ tend to $29/25$.  In the
complementary direction, for any CTLN on $n\geq3$ vertices, we prove that in
the parameter range
\[
 q\geq n-2-\frac{n-3}{2}\varepsilon,
\]
no nonclique support can satisfy both the fixed-point positivity and linear
stability conditions.
Consequently, throughout this range, every nondegenerate CTLN has
exactly its target-free cliques as supports of stable
fixed points.  In particular, this holds when
$\varepsilon\leq\delta/(\delta+n-2)$.  
\end{abstract}



\section{Introduction}

A combinatorial threshold-linear network, or CTLN, is a competitive
threshold-linear network determined by a loopless directed graph, two
connectivity parameters $\varepsilon,\delta$, and a positive uniform drive
$\theta$.  Their piecewise-linear dynamics range from fixed points to chaos
\cite{MorrisonDegeratuItskovCurto2024}, while their fixed-point supports admit
a graphical theory \cite{CurtoGenesonMorrison2019} extending the permitted-set
theory of symmetric threshold-linear networks
\cite{HahnloserSeungSlotine2003}.  Distinguished induced subgraphs also
predict nonstationary attractors \cite{ParmeleeMooreMorrisonCurto2022} and
support constructions for sequence and pattern generation
\cite{LondonoAlvarezMorrisonCurto2026}.

Every target-free clique supports a stable fixed point.  The conjecture that
these are the only stable fixed-point supports of a nondegenerate CTLN
\cite{MorrisonDegeratuItskovCurto2024} is known for symmetric graphs
\cite{CurtoMorrison2016}, oriented graphs
\cite{MorrisonDegeratuItskovCurto2024}, supports of size at most four, and
several other graph and parameter classes \cite{CurtoGenesonMorrison2024}.
It appears as Open Question~10 in \cite{CurtoMorrison2023}.  The analogous
minimal-support conjecture for general competitive threshold-linear networks
has recently been disproved \cite{Geneson2025}, but that construction is not
combinatorial.

For $\varepsilon,\delta>0$, write
\[
 q=\frac{\delta(1-\varepsilon)}{\varepsilon};
\]
the legal CTLN parameter range is $q>1$.  Theorem~\ref{thm:six-counterexample}
gives a fixed six-vertex graph for which, at
$\delta=29\varepsilon/25$ and every sufficiently small $\varepsilon>0$,
the CTLN is nondegenerate and has a stable fixed point with nonclique full
support.  Thus the target-free clique conjecture is false. 

In the complementary direction, the all-graph sufficient bound in
\cite{CurtoGenesonMorrison2024} is $q>n^2-n-1$.  We replace it
by the linear bound
\begin{equation}\label{eq:intro-bound}
 q\geq n-2-\frac{n-3}{2}\varepsilon.
\end{equation}
It yields the simple sufficient condition
$\varepsilon\leq\delta/(\delta+n-2)$, improving the sufficient scale at
fixed $\delta$ from order $\delta/n^2$ to order $\delta/n$ as
$n\to\infty$ with $\delta$ fixed.

\section{Networks and support criteria}\label{sec:setup}

For every positive integer $d$, write $[d]=\{1,\ldots,d\}$.  Fix a
positive integer $n$, and let $G$ be a directed graph on $[n]$, with no
self-loops and at most one edge in each
ordered direction between distinct vertices.  Edges in both directions are
allowed.  The CTLN associated with $G$ has state
$x(t)\in\mathbb R_{\geq0}^n$, weight matrix $W\in\mathbb R^{n\times n}$,
and dynamics
\begin{equation}\label{eq:dynamics}
 \dot x_i=-x_i+
 \left[\sum_{j=1}^nW_{ij}x_j+\theta\right]_+,
 \qquad i\in[n],\qquad \theta>0,
\end{equation}
where $[u]_+=\max\{u,0\}$ and
\[
 W_{ij}=
 \begin{cases}
 0,&i=j,\\
 -1+\varepsilon,&j\to i,\\
 -1-\delta,&j\not\to i,
 \end{cases}
\]
at parameters
\begin{equation}\label{eq:legal}
 \delta>0,
 \qquad
 0<\varepsilon<\frac{\delta}{1+\delta}.
\end{equation}
We call \eqref{eq:legal} the \emph{legal range}.

A \emph{fixed point} is a state $x^*$ at which the right side of
\eqref{eq:dynamics} vanishes.  Its \emph{support} is
\[
 \operatorname{supp}(x^*)=\{i\in[n]:(x^*)_i>0\}.
\]
Because $\theta>0$, the zero state is not a fixed point; hence fixed-point
supports are nonempty.  Using the Euclidean norm, we call a fixed point
\emph{stable} if there are
constants $r,M,\alpha>0$ such that every solution with
$\lVert x(0)-x^*\rVert<r$ satisfies
\begin{equation}\label{eq:stability-definition}
 \lVert x(t)-x^*\rVert
 \leq M e^{-\alpha t}\lVert x(0)-x^*\rVert
 \qquad(t\geq0).
\end{equation}
Since $[\theta u]_+=\theta[u]_+$ for $\theta>0$, the substitution
$x=\theta y$ reduces \eqref{eq:dynamics} to unit drive.  Thus positive drive
scaling changes fixed-point values but not their supports or stability.

A \emph{clique} is a vertex set joined in both directions.  A vertex
outside a clique is a \emph{target} if it receives an edge from every
vertex in the clique.  A clique is \emph{target-free} if it has no target.

Fix a nonempty candidate support $\sigma=\{v_1,\ldots,v_m\}\subseteq[n]$
and fix the displayed ordering.  Its \emph{nonedge matrix} is the
$m\times m$ zero-one matrix $N$ defined, for $a,b\in[m]$, by
\[
 N_{aa}=0,
 \qquad
 N_{ab}=1\ \Longleftrightarrow\ v_b\not\to v_a
 \quad(a\ne b).
\]
Set
\begin{equation}\label{eq:parameters}
 t=\frac{1-\varepsilon}{\varepsilon},
 \qquad
 \beta=1+\frac{\delta}{\varepsilon},
 \qquad
 q=\frac{\delta(1-\varepsilon)}{\varepsilon}.
\end{equation}
For each $d\geq1$, let $I_d$ be the $d\times d$ identity matrix, let
$\one_d\in\mathbb R^d$ be the all-ones column vector, and put
$J_d=\one_d\one_d^T$.  We also write
\[
 \one_d^\perp=\{u\in\mathbb R^d:\one_d^Tu=0\}.
\]
We omit a dimension subscript only when the
dimension is forced unambiguously.  Write $(I_n-W)_\sigma$ for the
principal submatrix indexed by the displayed ordering of $\sigma$, and set
\begin{equation}\label{eq:K}
 A_\sigma:=(I_n-W)_\sigma,
 \qquad
 K:=\varepsilon^{-1}A_\sigma=I_m+\beta N+tJ_m.
\end{equation}
Thus $K$ is the active principal matrix with the positive scalar
$\varepsilon$ divided out.  The legal condition is equivalent to $q>1$,
and also to $t(\beta-2)>1$.

We state the standard matrix identities used below in their exact forms.  For
$A\in\mathbb C^{d\times d}$, $b\in\mathbb C^d$, $i\in[d]$, and
$S\subseteq[d]$, let $\operatorname{adj}(A)$ be the transpose of the cofactor
matrix; write $A[i\leftarrow b]$ for $A$ with column $i$ replaced by $b$, and
$A_S$ for the principal submatrix indexed by $S$.  If $A$ is nonsingular,
Cramer's formula is
\begin{equation}\label{eq:cramer}
 (A^{-1}b)_i=\frac{\det A[i\leftarrow b]}{\det A}.
\end{equation}
For $u,v\in\mathbb C^d$, the rank-one determinant identities are
\begin{align}
 \det(A+uv^T)&=\det A+v^T\operatorname{adj}(A)u,
   \label{eq:rank-one-determinant}\\
 \det(A+uv^T)&=\det A\bigl(1+v^TA^{-1}u\bigr)
   \quad(A\text{ nonsingular}).
   \label{eq:rank-one-determinant-invertible}
\end{align}
For $\alpha\in\mathbb C$, $b,c\in\mathbb C^d$, and
$M\in\mathbb C^{d\times d}$, the bordered identity, valid without assuming
that $M$ is nonsingular, is
\begin{equation}\label{eq:block-determinant}
 \det\begin{pmatrix}\alpha&-b^T\\-c&M\end{pmatrix}
 =\alpha\det M-b^T\operatorname{adj}(M)c.
\end{equation}
Writing
$\det(zI_d+A)=z^d+c_1(A)z^{d-1}+\cdots+c_d(A)$ as a polynomial in the
indeterminate $z$, we have, for $1\leq j\leq d$,
\begin{equation}\label{eq:coefficient-minors}
 c_j(A)=\sum_{\substack{S\subseteq[d]\\|S|=j}}\det A_S,
 \qquad c_1(A)=\operatorname{tr}A,
\end{equation}
where $\operatorname{tr}A=\sum_{i=1}^dA_{ii}$.  Finally, if
$\lambda_1,\ldots,\lambda_d$ are the eigenvalues of $A$, counted with
algebraic multiplicity, then
\begin{equation}\label{eq:spectral-facts}
 \det(zI_d+A)=\prod_{j=1}^d(z+\lambda_j),
 \qquad \operatorname{tr}A=\sum_{j=1}^d\lambda_j.
\end{equation}
For real $A$, nonreal eigenvalues occur in conjugate pairs, and the algebraic
multiplicity of an eigenvalue is at least the dimension of its eigenspace.
The identities above are from
\cite{HornJohnson2013}.

A real square matrix is called \emph{stable} if every eigenvalue has
positive real part.  We use the same word for a linear endomorphism when
one, equivalently every, representing matrix is stable.  A real
$d\times d$ matrix $A$ is \emph{permitted} (for uniform input) if it is
nonsingular and $A^{-1}\one_d>0$ entrywise.  We call $\sigma$ permitted
when $K$ is permitted.  In that case define
\begin{equation}\label{eq:restricted-fp}
 x_\sigma^*=\frac{\theta}{\varepsilon}K^{-1}\one_m,
 \qquad x_{[n]\setminus\sigma}^*=0.
\end{equation}
The standard fixed-point equations say that this candidate is a fixed point
exactly when, for every $k\notin\sigma$,
\begin{equation}\label{eq:outside}
 \theta+\sum_{j\in\sigma}W_{kj}(x_\sigma^*)_j\leq0.
\end{equation}
Conversely, if $A_\sigma$ is nonsingular, every fixed point with support
$\sigma$ has the form \eqref{eq:restricted-fp} and satisfies
\eqref{eq:outside} \cite{CurtoGenesonMorrison2019}.

We say that $\sigma$ satisfies the \emph{algebraic stable-support
conditions} if it is permitted, $K$ is stable, and all inequalities
\eqref{eq:outside} hold. 

A CTLN is \emph{nondegenerate} if, for every nonempty subset
$\rho\subseteq[n]$, the principal matrix $(I_n-W)_\rho$ is nonsingular
and, for every $i\in\rho$, the determinant obtained by replacing its
column indexed by $i$ with $\one_{|\rho|}$ is nonzero
\cite{CurtoGenesonMorrison2024}.  We call the latter
determinants \emph{uniform-input column-replacement determinants}.  In a
nondegenerate CTLN, every outside inequality of a fixed point is strict.
Indeed, if equality held at $u\notin\sigma$, then the vector $z$, indexed
by $\sigma\cup\{u\}$ and defined by $z_j=(x_\sigma^*)_j$ for
$j\in\sigma$ and $z_u=0$, would solve
\[
 (I_n-W)_{\sigma\cup\{u\}}z=\theta\one_{m+1}.
\]
Order $\sigma\cup\{u\}$ with $u$ last, let
$A=(I_n-W)_{\sigma\cup\{u\}}$, and set
$A^{(u)}=A[m+1\leftarrow\one_{m+1}]$.  Formula \eqref{eq:cramer} gives
\[
 z_u=\theta\frac{\det A^{(u)}}{\det A}.
\]
Since $z_u=0$, $\theta>0$, and $\det A\ne0$, this would give
$\det A^{(u)}=0$, contrary to nondegeneracy.

For a nondegenerate CTLN, the algebraic stable-support
conditions are equivalent to stability of the corresponding fixed point \cite{CurtoGenesonMorrison2024}. We say that a CTLN has the \emph{algebraic target-free clique
classification} if the supports satisfying the algebraic stable-support
conditions are exactly its target-free cliques.  For a nondegenerate CTLN,
this is equivalent to the \emph{dynamical target-free clique
classification}: its target-free cliques are exactly the supports of its
stable fixed points.

\begin{lemma}\label{lem:activity}
Let $d\geq1$, let $N\in\mathbb R^{d\times d}$ be entrywise nonnegative
with zero diagonal, let $\beta,t\geq0$, and put
$C=I_d+\beta N$ and $K=C+tJ_d$.  Then
\[
 K\text{ is permitted}
 \quad\Longleftrightarrow\quad
 C\text{ is permitted}.
\]
When these conditions hold,
\begin{equation}\label{eq:rank-one-activity}
 K^{-1}\one_d=
 \frac{C^{-1}\one_d}{1+t\one_d^TC^{-1}\one_d}.
\end{equation}
\end{lemma}

\begin{proof}
Suppose $x=K^{-1}\one_d>0$.  Since
$Cx=a\one_d$, where $a=1-t\one_d^Tx$, and $C\geq I_d$ entrywise, we have
$a\one_d=Cx\geq x>0$.  If $C$ were singular, a nonzero $z^T$ with
$z^TC=0$ would give $0=z^TCx=a z^T\one_d$, hence $z^TK=0$, contradicting
the nonsingularity of $K$.  Thus $C^{-1}\one_d=x/a>0$.

Conversely, let $y=C^{-1}\one_d>0$.  The denominator
$1+t\one_d^Ty$ is positive, and the rank-one identity
\eqref{eq:rank-one-determinant-invertible}, with
$u=t\one_d$ and $v=\one_d$, gives
\[
 \det K
 =\det C\,(1+t\one_d^TC^{-1}\one_d)\ne0,
\]
while direct multiplication gives
$K[y/(1+t\one_d^Ty)]=\one_d$.  This proves
\eqref{eq:rank-one-activity} and the converse.
\end{proof}

We will use the standard target-free clique rule
\cite{CurtoGenesonMorrison2024}.  For completeness, its entire
calculation in the present notation is as follows.  If $\sigma$ is a clique
of size $s$, then
\[
 (I_n-W)_\sigma=\varepsilon I_s+(1-\varepsilon)J_s.
\]
Its eigenvalues are $\varepsilon$, with multiplicity $s-1$, and
$s-(s-1)\varepsilon$, and the restricted fixed point has constant coordinate
\[
 c=\frac{\theta}{s-(s-1)\varepsilon}>0.
\]
If $u\notin\sigma$ receives edges from all but $r$ clique vertices, its input
is
\begin{equation}\label{eq:clique-input}
 \theta+\sum_{j\in\sigma}W_{uj}c
 =\frac{\theta[\varepsilon-r(\varepsilon+\delta)]}
 {s-(s-1)\varepsilon}.
\end{equation}
This is positive for $r=0$ and strictly negative for $r\geq1$.  Hence a
clique satisfies the algebraic stable-support conditions exactly when it is
target-free, and then its fixed point is stable.

Two elementary observations will be used below.  First, if a real matrix
$A$ is stable, every coefficient of $\det(zI+A)$ is positive: each positive
real eigenvalue contributes $z+a$, and each conjugate pair $a\pm\ii b$
contributes $z^2+2az+a^2+b^2$, with $a>0$.  Second, a nonzero zero-one
matrix $N$ of order at most two with zero diagonal cannot make
$K=I+\beta N+tJ$ stable at legal parameters.  The order-one case is
vacuous; up to transposition, the order-two possibilities are
\[
 \begin{pmatrix}0&1\\0&0\end{pmatrix}
 \quad\text{or}\quad
\begin{pmatrix}0&1\\1&0\end{pmatrix}.
\]
Writing $L=\one_2^TN\one_2$ and letting $w$ indicate that both off-diagonal
entries are one, direct expansion gives
\[
 \det K=1+2t-\beta tL-\beta^2w.
\]
For $L=1$ this is $1-t(\beta-2)<0$ by legality; for $L=2$ it is
$1+2t-2\beta t-\beta^2<0$ because $\beta>2$.  Both contradict the positive
constant coefficient required by stability.

\section{An explicit six-neuron counterexample}\label{sec:six}

For $m\geq1$, set
\[
 P_m=I_m-\frac1mJ_m.
\]
Then $P_m^T=P_m$, $P_m^2=P_m$, $P_m\one_m=0$, and $P_mu=u$ for
$u\in\one_m^\perp$; hence $P_m$ is the orthogonal projection onto
$\one_m^\perp$.

\begin{lemma}\label{lem:spectral-limit}
Let $m\geq2$ and let $C\in\mathbb R^{m\times m}$.  If the endomorphism
$(P_mC)|_{\one_m^\perp}:\one_m^\perp\to\one_m^\perp$ is stable, then
$C+tJ_m$ is stable for all sufficiently large $t$.
\end{lemma}

\begin{proof}
Put $u=\one_m/\sqrt m$, and let
$Q\in\mathbb R^{m\times(m-1)}$ have as its columns an orthonormal basis of
$\one_m^\perp$.  Since $P_m=QQ^T$ and $J_m=muu^T$, the stable matrix
$D=Q^TCQ$ represents $(P_mC)|_{\one_m^\perp}$, and the matrix of
$A_t=C+tJ_m$ in the orthonormal basis $[u,Q]$ is
\[
 \begin{pmatrix}
 mt+a&b^T\\
 c&D
 \end{pmatrix},
 \qquad
 a=u^TCu,\qquad b^T=u^TCQ,\qquad c=Q^TCu.
\]
The exact bordered identity \eqref{eq:block-determinant}, with
$\alpha=\lambda-mt-a$ and $M=\lambda I_{m-1}-D$, gives
\begin{equation}\label{eq:limit-characteristic}
\begin{split}
 \det(\lambda I_m-A_t)
 ={}&(\lambda-mt-a)\det(\lambda I_{m-1}-D)\\
 &-b^T\operatorname{adj}(\lambda I_{m-1}-D)c.
\end{split}
\end{equation}
Define
\[
 p_t(\lambda)=\frac{-1}{mt}\det(\lambda I_m-A_t),
 \qquad
 p(\lambda)=\det(\lambda I_{m-1}-D).
\]
Equation~\eqref{eq:limit-characteristic} gives the explicit formula
\begin{equation}\label{eq:limit-uniform}
 p_t(\lambda)
 =\left(1+\frac{a-\lambda}{mt}\right)p(\lambda)
 +\frac1{mt}b^T\operatorname{adj}(\lambda I_{m-1}-D)c.
\end{equation}
Thus $p_t\to p$ uniformly on compact sets.  Choose pairwise disjoint closed
disks, each contained in the open right half-plane and containing one
distinct eigenvalue of $D$ in its interior, so that their union contains the
spectrum with algebraic multiplicity.  Uniform convergence gives, for all
sufficiently large $t$,
\[
 |p_t(\lambda)-p(\lambda)|<|p(\lambda)|
\]
on these boundaries.  In the exact form used here, Rouch\'e's theorem says
that holomorphic $f,g$ near a closed disk have the same number of interior
zeros, counted with multiplicity, if $|f-g|<|g|$ on its boundary
\cite{Ahlfors1979}.  Applied to each disk, it shows that
$m-1$ eigenvalues $\mu_{1,t},\ldots,\mu_{m-1,t}$ of $A_t$ lie in their union
and remain bounded.

Let $\nu_t$ be the remaining eigenvalue of $A_t$.  The trace identity in
\eqref{eq:spectral-facts} and $\operatorname{tr}J_m=m$ give
\[
 \nu_t+\sum_{j=1}^{m-1}\mu_{j,t}
 =\operatorname{tr}A_t=mt+\operatorname{tr}C.
\]
Hence $\operatorname{Re}\nu_t\to+\infty$.  All eigenvalues of $C+tJ_m$ are
therefore in the open right half-plane for sufficiently large $t$.
\end{proof}

\Needspace{18\baselineskip}
\begin{theorem}\label{thm:six-counterexample}
Let
\begin{equation}\label{eq:six-N}
N=
\begin{pmatrix}
0&0&0&0&1&1\\
0&0&0&0&0&1\\
0&0&0&0&0&1\\
1&1&0&0&1&0\\
0&0&1&1&0&0\\
0&0&0&1&1&0
\end{pmatrix}.
\end{equation}
There is an $\varepsilon_0>0$ such that, for every
$0<\varepsilon<\varepsilon_0$ and every $\theta>0$, the CTLN with full
nonedge matrix $N$ and
\[
 \delta=\frac{29}{25}\varepsilon
\]
is legal and nondegenerate, its full support is permitted, and the
corresponding fixed point is stable.  The full
support is not a clique, and
\[
 q=\frac{29}{25}(1-\varepsilon)\longrightarrow\frac{29}{25}
 \qquad(\varepsilon\downarrow0).
\]
\end{theorem}

\begin{proof}
Set
\[
 \beta=\frac{54}{25},
 \qquad
 C=I_6+\beta N.
\]
Put $B=25C$.  In the $3+3$ block partition, subtracting $25^{-1}$ times the
lower-left block times the first block row from the second block row gives
the exact determinant reduction
\[
 \det B=\det\begin{pmatrix}
 625&-1566&-5832\\1350&625&-2916\\1350&1350&625
 \end{pmatrix}
 =4{,}482{,}473{,}725=25\cdot179{,}298{,}949.
\]
Since $C=B/25$, this and the displayed matrix--vector product give the
permission certificate
\begin{equation}\label{eq:six-activity}
 \det C=\frac{179298949}{9765625},
 \qquad
 v=
 \begin{pmatrix}
 21597025\\34234375\\34234375\\46065775\\5850625\\67159525
 \end{pmatrix}>0,
 \qquad
 Cv=179298949\one_6.
\end{equation}
Thus $C^{-1}\one_6=v/179298949>0$.
It follows from Lemma~\ref{lem:activity} that $C+tJ_6$ is permitted for
every $t\geq0$.

We next prove that $(P_6C)|_{\one_6^\perp}$ is stable.  Writing
$\mathbf e_1,\ldots,\mathbf e_6$ for the standard basis of $\mathbb R^6$,
direct multiplication in the basis
$\mathbf e_1-\mathbf e_6,\ldots,\mathbf e_5-\mathbf e_6$ of
$\one_6^\perp$ gives the following matrix for this restriction:
\[
 D=\frac1{25}M,
 \qquad
 M=\begin{pmatrix}
 -11&-36&-36&-45&0\\
 -36&-11&-36&-45&-54\\
 -36&-36&-11&-45&-54\\
 72&72&18&34&54\\
 18&18&72&63&25
 \end{pmatrix}.
\]
Here is a hand-checkable certificate for the characteristic polynomial.
Let $\widehat e_1,\ldots,\widehat e_5$ be the standard basis of
$\mathbb R^5$ and let
\[
 S=[(-1,1,0,0,0)^T,\widehat e_2,\widehat e_3,
       \widehat e_4,\widehat e_5].
\]
Then $\det S=-1$, and direct multiplication gives
\[
 S^{-1}MS=
 \begin{pmatrix}
 25&36&36&45&0\\
 0&-47&-72&-90&-54\\
 0&-36&-11&-45&-54\\
 0&72&18&34&54\\
 0&18&72&63&25
 \end{pmatrix}
 =\begin{pmatrix}25&*\\0&Q\end{pmatrix},
\]
where $Q$ is the displayed lower-right $4\times4$ block.  Applying the
principal-minor formula \eqref{eq:coefficient-minors} to this explicit block
gives
\[
 (c_1(Q),c_2(Q),c_3(Q),c_4(Q))=(1,4101,3235,1394278).
\]
The block triangular form therefore gives the exact factorization
\[
 \det(xI_5+M)
 =(x+25)(x^4+x^3+4101x^2+3235x+1394278).
\]
Substituting $x=25z$ and dividing by $25^5$ yields
\begin{equation}\label{eq:six-characteristic}
\det(zI_5+D)
=\frac{z+1}{390625}
 \left(390625z^4+15625z^3+2563125z^2
       +80875z+1394278\right).
\end{equation}
\Needspace{13\baselineskip}
We use the following exact quartic Routh--Hurwitz criterion.  For real
coefficients and $a_4>0$, every zero of
$a_4z^4+a_3z^3+a_2z^2+a_1z+a_0$ lies in the open left half-plane if and
only if
\[
 H_1=a_3,
 \quad H_2=a_3a_2-a_4a_1,
 \quad H_3=a_3a_2a_1-a_4a_1^2-a_3^2a_0,
 \quad H_4=a_0H_3
\]
are all positive \cite{Gantmacher1959}.  For the quartic in
\eqref{eq:six-characteristic}, $a_4=390625>0$, and exact substitution gives
\[
 (H_1,H_2,H_3,H_4)
 =(15625,8457031250,343562500000000,
 479021635375000000000).
\]
All four numbers are positive, so the quartic zeros lie in the open left
half-plane; the remaining factor $z+1$ has zero $-1$.  By
\eqref{eq:spectral-facts}, the zeros of $\det(zI_5+D)$ are the negatives of
the eigenvalues of $D$.  Hence $D$ is stable.  Lemma~\ref{lem:spectral-limit}
now shows that $K=C+tJ_6$ is stable for all sufficiently large $t$.

For nonempty $\rho\subseteq[6]$, let $N_\rho$ be the corresponding principal
submatrix of $N$, put $k=|\rho|$, and define
\[
 B_\rho=25I_k+54N_\rho,
 \qquad
 C_\rho=\frac1{25}B_\rho.
\]
Modulo two, $B_\rho\equiv I_k$, so
$\det B_\rho\equiv1\pmod 2$.  If $e_1,\ldots,e_k$ are the standard coordinate
vectors, replacing column $i$ by
$\one_k=\sum_{j=1}^k e_j$ also gives determinant one modulo two, since every
term except $e_i$ repeats another identity column.  Hence all these
determinants are nonzero.  Scaling gives
\[
 \det C_\rho=25^{-k}\det B_\rho,
\]
and each column-replacement determinant for $C_\rho$ is $25^{1-k}$ times
its counterpart for $B_\rho$.

For $K_\rho=C_\rho+tJ_k$, apply \eqref{eq:rank-one-determinant} with
$A=C_\rho$, $u=t\one_k$, and $v=\one_k$ to obtain
\begin{equation}\label{eq:principal-affine}
 \det K_\rho
 =\det C_\rho+t\one_k^T\operatorname{adj}(C_\rho)\one_k.
\end{equation}
The right side is affine in $t$ with nonzero constant term.  After column
$i$ of $K_\rho$ is replaced by $\one_k$, subtracting $t$ times that column
from every other column leaves exactly the corresponding nonzero replacement
determinant of $C_\rho$, independent of $t$.  Finiteness of the set of
$\rho$ therefore gives $T_1$ such that all required determinants of every
$K_\rho$ are nonzero for $t>T_1$.

Choose $T$ larger than $T_1$, the stability threshold supplied by
Lemma~\ref{lem:spectral-limit}, and $25/4$.  For $t>T$, set
\begin{equation}\label{eq:parameter-specialization}
 \varepsilon=\frac1{t+1},
 \qquad
 \delta=\frac{29}{25}\varepsilon.
\end{equation}
Then \eqref{eq:parameters} gives $\beta=54/25$ and
\[
 q=\frac{29t}{25(t+1)}>1.
\]
So the parameters are legal.  Since
$(I_6-W)_\rho=\varepsilon K_\rho$, its principal determinant scales by
$\varepsilon^{|\rho|}$ and each replacement determinant by
$\varepsilon^{|\rho|-1}$.  Nondegeneracy follows.

The full-support fixed point is positive by \eqref{eq:six-activity} and
Lemma~\ref{lem:activity}.  In a neighborhood of this positive full-support
fixed point, the Jacobian is $-\varepsilon K$, whose eigenvalues have
negative real part.  The fixed point is therefore stable.  Since $N\ne0$,
its support is not a clique.  Finally, taking
$\varepsilon_0=1/(T+1)$ proves the uniform assertion for every
$0<\varepsilon<\varepsilon_0$, and the displayed formula for $q$ follows
directly from \eqref{eq:parameter-specialization}.
\end{proof}

\Needspace{24\baselineskip}
\section{A linear sufficient range}\label{sec:linear}

\begin{theorem}\label{thm:linear}
Let $\delta>0$ and $0<\varepsilon<\delta/(1+\delta)$, and define
$t,\beta,q$ by \eqref{eq:parameters}.  Let $m\geq3$ and let
$N\in\{0,1\}^{m\times m}$ be nonzero with zero diagonal.  If
$K=I_m+\beta N+tJ_m$ is permitted and stable, then
\begin{equation}\label{eq:support-bound}
 q<m-2-\frac{m-3}{2}\varepsilon.
\end{equation}
Consequently, at these parameters, every CTLN on a graph with $n\geq3$
vertices has the algebraic target-free clique classification whenever
\begin{equation}\label{eq:global-bound}
 q\geq n-2-\frac{n-3}{2}\varepsilon.
\end{equation}
For nondegenerate CTLNs, the same condition gives the dynamical
target-free clique classification.
\end{theorem}

\begin{proof}
Legality gives $\beta>2$ and $t>0$.  By Lemma~\ref{lem:activity}, the
vector $y=(I_m+\beta N)^{-1}\one_m$ is positive and satisfies
\begin{equation}\label{eq:y}
 y_i+\beta\sum_{j=1}^mN_{ij}y_j=1
 \qquad(1\leq i\leq m).
\end{equation}
Let
\[
 H=\{i\in[m]:\text{row }i\text{ or column }i
       \text{ of }N\text{ is nonzero}\},
 \qquad h=|H|,
 \qquad r=m-h,
\]
and call the indices in $H$ \emph{nonisolated} and those in
$[m]\setminus H$ \emph{isolated}.  Put $L=\one_m^TN\one_m$.  Since $N$
is nonzero, $L>0$.  If column
$j$ is nonzero, then $N_{ij}=1$ for some $i$,
and \eqref{eq:y} gives $\beta y_j<1$.  Row $j$ must then be nonzero,
because a zero row would give $y_j=1$.  Thus every one of the $h$
nonisolated vertices has a nonzero row, and consequently $L\geq h$.

In fact,
\begin{equation}\label{eq:L-lower}
 L\geq h+1.
\end{equation}
Suppose instead that $L=h$.  Every nonisolated row then contains exactly
one entry equal to one, and its column is also nonisolated.  The resulting
map $f:H\to H$, defined by letting $f(i)$ be the unique $j$ with
$N_{ij}=1$, contains a cycle of distinct vertices
$i_1,\ldots,i_\ell$ with $\ell\geq2$ and
\[
 N_{i_s,i_{s+1}}=1
 \quad(1\leq s\leq\ell),
\]
where subscripts are cyclic.  Put
\[
 \zeta=\exp\left(
 \frac{2\pi\ii\,\lfloor\ell/2\rfloor}{\ell}
 \right).
\]
Then $\zeta\ne1$ and $\operatorname{Re}\zeta\leq-1/2$ (for odd $\ell$,
the real part is $-\cos(\pi/\ell)$).  Define $a^T\in\mathbb C^{1\times m}$
by
\[
 a_{i_s}=\zeta^{-(s-1)},
 \qquad
 a_i=0\quad\text{off the cycle}.
\]
Each cycle row has its unique one in the next cycle column, so
\[
 a^TN=\zeta a^T,
 \qquad
 a^T\one_m=\sum_{s=0}^{\ell-1}\zeta^{-s}=0.
\]
Consequently,
\[
 a^TK=(1+\beta\zeta)a^T.
\]
Since $a^T\ne0$, transposing shows that $1+\beta\zeta$ is an eigenvalue of
$K$.  Its real part is at most $1-\beta/2<0$, contradicting stability.  This
proves \eqref{eq:L-lower}.

Let $w$ be the number of unordered pairs $\{i,j\}$ for which
$N_{ij}=N_{ji}=1$.  The determinant of the principal two-by-two submatrix
of $K$ indexed by $i<j$ is
\[
 (1+t)^2-(t+\beta N_{ij})(t+\beta N_{ji})
 =1+2t-\beta t(N_{ij}+N_{ji})-\beta^2N_{ij}N_{ji}.
\]
Summing over all pairs gives
\begin{equation}\label{eq:c2}
 c_2(K)=\binom m2(1+2t)-\beta tL-\beta^2w.
\end{equation}

Suppose first that $r\geq1$, and define
\[
 U=\{x\in\mathbb R^m:x_i=0\text{ for }i\in H,
                    \ \one_m^Tx=0\}.
\]
The space of vectors supported on the $r$ isolated indices has dimension
$r$, so $\dim U=r-1$; moreover $Nx=J_mx=0$ and $Kx=x$ on $U$.  Hence
\[
 \det(zI_m+K)=(z+1)^{r-1}p(z),
\]
where $p$ is real, monic, of degree $h+1$, and has positive coefficients by
stability and the factor observation above.  Write
\[
 p(z)=z^{h+1}+a_1z^h+a_2z^{h-1}+\cdots.
\]
Since $h\geq2$, $a_2$ is present.  Comparing the first two coefficients and
using \eqref{eq:c2} gives
\[
 a_1=h+1+mt,
 \qquad
 a_2=c_2(K)-(r-1)a_1-\binom{r-1}{2},
\]
and therefore
\begin{equation}\label{eq:a2}
 0<a_2=\binom{h+1}{2}+hmt-\beta tL-\beta^2w.
\end{equation}
The parameter identity
\begin{equation}\label{eq:beta-t}
 \beta t=q(t+1)+t
\end{equation}
follows from $1/\varepsilon=t+1$ and $q=\delta t$.  Together with
$\varepsilon=1/(t+1)$, inequality \eqref{eq:a2} gives
\[
 [q(t+1)+t]L<\binom{h+1}{2}+hmt,
\]
where dropping $-\beta^2w$ only enlarges the right side.  Dividing by
$L(t+1)>0$ and using $t/(t+1)=1-\varepsilon$ gives
\begin{equation}\label{eq:q-intermediate}
 q<(1-\varepsilon)\left(\frac{hm}{L}-1\right)
 +\varepsilon\frac{\binom{h+1}{2}}{L}.
\end{equation}
By $h\leq m-1$ and \eqref{eq:L-lower},
\[
 \frac{hm}{L}-1
 \leq\frac{hm}{h+1}-1\leq m-2,
 \qquad
 \frac{\binom{h+1}{2}}{L}
 \leq\frac h2\leq\frac{m-1}{2}.
\]
Substitution in \eqref{eq:q-intermediate} gives
\[
 q<(1-\varepsilon)(m-2)+\varepsilon\frac{m-1}{2}
   =m-2-\frac{m-3}{2}\varepsilon,
\]
which proves
\eqref{eq:support-bound} when $r\geq1$.

Suppose now that $r=0$.  Then $h=m$ and \eqref{eq:L-lower} gives
$L\geq m+1$.  Stability gives $c_2(K)>0$ by the coefficient observation
above.
Using \eqref{eq:c2}, dropping $-\beta^2w$, and applying
\eqref{eq:beta-t} gives
\[
 [q(t+1)+t]L<\binom m2(1+2t).
\]
Division by $L(t+1)$, together with
$(1+2t)/(t+1)=2-\varepsilon$ and
$t/(t+1)=1-\varepsilon$, gives
\begin{equation}\label{eq:no-isolates-bound}
 q<\frac{\binom m2(2-\varepsilon)}{L}-(1-\varepsilon)
 \leq\frac{m(m-1)(2-\varepsilon)}{2(m+1)}-(1-\varepsilon).
\end{equation}
The right side of \eqref{eq:support-bound} exceeds the last expression by
\[
 \frac{(2-\varepsilon)(m-1)}{2(m+1)}>0.
\]
This proves \eqref{eq:support-bound} in the remaining case.

The right side of \eqref{eq:support-bound} increases with $m$ by
$1-\varepsilon/2>0$.  Thus \eqref{eq:global-bound} excludes nonclique
supports of sizes three through $n$, and the order-two determinant
calculation above excludes the smaller cases.  The clique calculation says
that the remaining supports are exactly the target-free cliques.  This
proves the algebraic classification; for nondegenerate CTLNs, the support
criterion in Section~\ref{sec:setup} gives the dynamical conclusion.
\end{proof}

\begin{corollary}\label{cor:epsilon}
Let $n\geq3$, $\delta>0$, and
$0<\varepsilon<\delta/(1+\delta)$.  Every CTLN at these parameters has the
algebraic target-free clique classification, and every nondegenerate CTLN
has the dynamical target-free clique classification, if
\begin{equation}\label{eq:exact-epsilon}
 \varepsilon\leq
 \frac{2\delta}
 {\delta+n-2+\sqrt{\delta^2+2\delta+(n-2)^2}}.
\end{equation}
In particular, either conclusion holds if
\begin{equation}\label{eq:simple-epsilon}
 \varepsilon\leq\frac{\delta}{\delta+n-2}.
\end{equation}
\end{corollary}

\begin{proof}
Substituting $q=\delta(1-\varepsilon)/\varepsilon$ into
\eqref{eq:global-bound} gives
\begin{equation}\label{eq:quadratic}
 \frac{n-3}{2}\varepsilon^2
 -(\delta+n-2)\varepsilon+\delta\geq0.
\end{equation}
For $n=3$, this reduces to
$\varepsilon\leq\delta/(\delta+1)$, already implied by legality.  For
$n>3$, the discriminant is
$(\delta+n-2)^2-2\delta(n-3)=\delta^2+2\delta+(n-2)^2$; the quadratic formula
and rationalization give \eqref{eq:exact-epsilon} as the smaller root.  The
larger root exceeds $(\delta+n-2)/(n-3)>1$, while legal
$\varepsilon$ satisfies $0<\varepsilon<1$.  Hence
\eqref{eq:exact-epsilon} implies \eqref{eq:quadratic}.

Finally, \eqref{eq:simple-epsilon} is equivalent to $q\geq n-2$, which
implies \eqref{eq:global-bound}.
\end{proof}

\section*{Acknowledgments}

Codex with GPT-5.6 and Claude Code with Opus 4.8 and Fable 5 were used
for proof exploration, proof criticism, exposition, and revision.

\end{document}